\documentclass[11pt,reqno]{amsart}

\usepackage{enumerate,url,amssymb,  mathrsfs}

\usepackage{graphicx}
\usepackage{esint,comment}
\pdfoutput=1

\newtheorem{theorem}{Theorem}[section]

\newtheorem*{lemma*}{Lemma}

\newtheorem{corollary}[theorem]{Corollary}

\theoremstyle{definition}

\theoremstyle{remark}

\numberwithin{equation}{section}

\def\XXint#1#2#3{{\setbox0=\hbox{$#1{#2#3}{\int}$}
\vcenter{\hbox{$#2#3$}}\kern-.5\wd0}}

\begin{document}
\baselineskip6mm
\vskip0.4cm
\title[Some transcendental equations on the Stieltjes cone]{Some transcendental equations on the Stieltjes cone}

\author[Filippo Giraldi]{Filippo Giraldi}
\address{School of Chemistry and
Physics, University of KwaZulu-Natal and National Institute for
Theoretical Physics, KwaZulu-Natal, Westville Campus, Durban 
4000, South Africa \\  
Gruppo Nazionale per la
Fisica Matematica (GNFM-INdAM)
c/o Istituto Nazionale di Alta Matematica Francesco Severi\\
Citta' Universitaria, Piazza Aldo Moro 5, 00185 Roma, Italy}
\email{giraldi.filippo@gmail.com,giraldi@ukzn.ac.za,filgi@libero.it}







\begin{abstract}
A general class of transcendental equations in complex domain is considered for functions belonging to the Stieltjes cone.
Under certain conditions 
each transcendental equation has no solution or one, at most, in the complex plane cut along the negative real axis. The unique solution is real valued and positive with an analytical bound. 
Particular cases consist in transcendental equations containing exponential, hyperbolic, power law, logarithmic and special functions. The present approach provides a simple way to prove that some special functions have no zero in certain sectors of the complex plane cut along the negative real axis. 
\end{abstract}
\allowdisplaybreaks

\maketitle


\section{Introduction}

The determination of the solutions of transcendental equations in complex domain is fundamental for the analysis of various types of equations. For example, autonomous differential-difference equations exhibit stable solutions if the zeros of the characteristic functions have negative real part \cite{Cooke1986}. In simple or multiple delay differential equations it is required to find the zeros of the exponential polynomials \cite{Ruan2003}. 
For certain integral equations of convoluted form \cite{IntEqPolyanin, BhattaBook2007} the solutions depends on the singularities appearing by Laplace transforming \cite{W}. 
Great interest is also devoted to the study of the zeros of entire and special functions with the most various applications, from relaxation-oscillation to fractional diffusion phenomena, to name a few \cite{SFnist,ZeroML1,ZeroML2,Csordas2006,Segura2002,Tikhonov2002,Tikhonov2005,kochubei2004,Mainardi1999,Mainardi1998}.

Here, we analyze a general class of transcendental
equations in complex domain that are obtained from the Steltjes transforms of nonnegative real valued functions \cite{W,GP,Hirsch1972, Berg1975,Berg2001,Berg2002,Berg2006,Pandey1972}. We intend to determine the conditions for the existence, uniqueness, reality, positivity of the solution and provide analytical bounds.

\section{The trascendental equation}\label{classEq}
The Stieltjes transform \cite{GP,W} of a function $f$ on $\mathbb{R}_+$ 
is defined 
as follows, 
\begin{eqnarray}
\mathcal{S}\left[f\right](z)=\int_0^{\infty}\frac{f\left(\zeta\right)}{\zeta+z}\, d\zeta, \label{Stieltjes}
\end{eqnarray}
if the above integral exists. The condition 

i) $f\in L_1^{{\rm loc} }
\left(\mathbb{R}^+\right)$ and $f\left(\zeta\right) =\mathcal{O}\left(\zeta^{-\delta}\right)$ for $\zeta\to+\infty$, where $\delta>0$,

\hspace{-0.9em}is sufficient for the existence of the Stieltjes transform.
If the Stieltjes transform of the function $f$ exists in $z_0 \in \mathbb{C}\backslash \left(-\infty\right. ,0]$ it exists
and is analytic \cite{GP} 
in the whole complex plane cut along the negative real axis, $\mathbb{C}\backslash \left(-\infty\right. ,0]$.

Consider the following class of transcendental equations in $\mathbb{C}\backslash \left(-\infty\right. ,0]$, 
\begin{eqnarray}
\mathcal{S}\left[\varphi\right](z)- b z - c=0, \label{EqS}
\end{eqnarray}
where $\varphi \in L_1^{{\rm loc} }\left(\mathbb{R}^+\right)$ and $\varphi\geq 0$ on $\mathbb{R}^+$, 
$b\geq 0$, $c \in \mathbb{R}$.
\begin{theorem}\label{Th1}
The transcendental equation (\ref{EqS}) has in $\mathbb{C}\backslash \left(-\infty\right. ,0]$  
one solution under each of the two following conditions,
\begin{eqnarray}
&& b>0, \hspace{1em} m_{\varphi}>0, \label{cond1}\\
&& b=0, \hspace{1em} m_{\varphi}>0,\hspace{1em} c>0, \label{cond2}
\end{eqnarray}
where the critical value $m_{\varphi}$ is defined as follows,
\begin{equation}
m_{\varphi}=\int_0^{\infty}\frac{\varphi\left(\zeta\right)}{\zeta}\, d\zeta-c. \label{m}
\end{equation}
No solution exists, otherwise. The unique solution is real valued and positive, $z=x_0>0$, and $\left(m_{\varphi}/b \right)>x_0>0$ if $b>0$.
The condition $m_{\varphi}=+\infty$ is included.
\end{theorem}

\begin{proof}
Let $x$ and $y$ be respectively the real and imaginary part of the 
variable $z$, i.e., $z=x+\imath y$, where $\imath$ is the imaginary unit.
The imaginary part of Eq. (\ref{Eq}) leads to the following equation,
\begin{equation}
y \left( b+ \int_0^{\infty}\frac{\varphi\left(\zeta\right)}{\left(\zeta+x\right)^2+ y^2}\,d\zeta\right)=0. \label{Eqy}
\end{equation}
The solution of Eq. (\ref{Eqy}) is $y=0$ since $b \geq 0$ and $\varphi\geq 0$ on $\mathbb{R}_+$, in addition to the trivial solution 
of vanishing function $\varphi$ and vanishing parameter $b$. 
Consequently, the real part of Eq. (\ref{Eq})
results in the following equation,
\begin{equation}
\int_0^{\infty}\frac{\varphi\left(\zeta\right)}{\left(\zeta+x\right)}\, d\zeta-c=b x, \label{Eqx}
\end{equation}
where $x>0$ since Eq. (\ref{EqS}) is defined in the domain $\mathbb{C}\backslash\left(-\infty\right. ,0]$.
 The left hand side of Eq. (\ref{Eqx}) is a strictly decreasing monotonic and continuous function \cite{GP,A} on $\mathbb{R_+}$, tends to $m_{\varphi}$
for $x\to 0^+$ and to $\left(-c^+\right)$ for $x\to +\infty$. Consequently, if $b>0$ one solutions exists
for $m_{\varphi}>0$. This solution is positive and bounded by the value $\left(\int_0^{\infty}
\varphi\left(\zeta\right)/\zeta\, d\zeta -c\right)/b$. No solution exists for $m_{\varphi}\leq 0$. If $b=0$ 
one solution exists for $\int_0^{\infty}
\varphi\left(\zeta\right)/\zeta\, d\zeta>c>0$, while no solution exists for $c\leq 0$ and for $c\geq\int_0^{\infty}
\varphi\left(\zeta\right)/\zeta\, d\zeta$. The present analysis holds also if $\int_0^{\infty}
\varphi\left(\zeta\right)/\zeta\, d\zeta =+\infty$.
\end{proof}

\subsection{Stieltjes representation}

Under certain conditions, a complex valued function 
is representable in the complex plane cut along the negative real axis as the Stieltjes transform of a nonnegative real valued function up to a nonnegative constant.
Following Ref. \cite{SeshuJMAA1962}, an analytic function $g$ is defined to be positive real if no singularity exists 
for $\Re\left\{z\right\}>0$,
if $g\left(\Re\left\{z\right\}\right)$ is real on $\mathbb{R}$ and if $\Re\left\{g(z)\right\}>0$ for $\Re\left\{z\right\}>0$. 
The class {\rm PRS} is defined as the set of functions $g$ such that $z g\left(z^2\right)$ is positive real \cite{SeshuJMAA1962}. 
Every function $g$ belonging to the class {\rm PRS} is representable \cite{SeshuJMAA1962} as a Stieltjes transform of a nonnegative real valued function, up to a nonnegative constant, in the whole complex plane cut along the negative real axis, $\mathbb{C}\backslash \left(-\infty\right. ,0]$.

Further conditions on the Stieltjes representation are reported in Ref. \cite{Hirsch1972,Berg1975,Berg2001,Berg2002,Berg2006}. If an analytic function $g$ on $\mathbb{C}\backslash \left(-\infty\right. ,0]$ fulfills the constraints
$\Im \left\{g(z)\right\}\leq 0$ for $\Im \left\{z\right\}>0$ and $g\left(\Re\left\{z\right\}\right)\geq 0$ for $\Re\left\{z\right\}>0$
the analytic extension of $g\left(\Re\left\{z\right\}\right)$ to $\mathbb{C}\backslash \left(-\infty\right. ,0]$, i.e., $g(z)$, is a Stieltjes transform of a nonnegative real valued function on $\mathbb{R}_+$ up to a nonnegative constant. This constant vanishes by requiring the following condition,
\begin{eqnarray}
\lim_{\Re\left\{z\right\}\to+\infty} g\left(\Re\left\{z\right\}\right)=0. \label{ginfty}
\end{eqnarray} The Stieltjes transforms of nonnegative functions on $\mathbb{R}_+$ represent a convex cone \cite{Hirsch1972,Berg1975,Berg2001,Berg2002,Berg2006} of the analytic functions on $\mathbb{C}\backslash \left(-\infty\right. ,0]$. 


In light of the Stieltjes representation described above, we focus on functions $\phi$ in $L_1^{{\rm loc} }\left(\mathbb{R}^+\right)$ that are real valued and nonnegative on $\mathbb{R}_+$. Let $\mathcal{L}^{(+)}_{\mathcal{S}}$ be the set of the functions $g$ on $\mathbb{C}\backslash \left(-\infty\right. ,0]$ that are the Stieltjes transforms of the nonnegative real valued functions $\phi$, 
which means $g(z)=\mathcal{S}\left[\phi\right](z)$. 
The following corollary holds.

\begin{corollary}
Consider the operator $\mathcal{T}_1:\phi\left(\zeta\right)\to \phi\left(a \zeta\right)$, where $a>0$, and the following
transcendental equation in $\mathbb{C}\backslash \left(-\infty\right. ,0]$,
\begin{eqnarray}
g\left( a z\right) - b z - c=0, \label{Eqg1}
\end{eqnarray}
where $g \in \mathcal{L}^{(+)}_{\mathcal{S}}$.
Since 
$\mathcal{S}\left[\mathcal{T}_1\left[\phi\right]\right](z)=g\left( a z\right) $, if $b>0$ Eq. (\ref{Eqg1})
has 
one solution for $\rho_{\phi}> 0$, where
\begin{eqnarray}
\rho_{\phi}=a\int_0^{\infty}\frac{\phi\left(\zeta\right)}{\zeta}\, d \zeta-c. \label{rho}
\end{eqnarray}
The solution is real valued, $z= \xi_1$, and $\left(\rho_{\phi}/b\right)>\xi_1>0 $. If $b=0$ Eq. (\ref{Eqg1})
has 
one solution for $\rho_{\phi} >c>0$. The solution is real valued and positive. No solution exists, 
otherwise.

If $\int_0^{\infty} \phi\left(\zeta\right)d \zeta<+\infty$, consider the operator 
$\mathcal{T}_2:\phi\left(\zeta\right)\to \zeta\phi\left( \zeta\right)$ and the following transcendental equation in $\mathbb{C}\backslash \left(-\infty\right. ,0]$,
\begin{eqnarray}
z g(z) + b z - c=0, \label{Eqg2}
\end{eqnarray}
where $g \in \mathcal{L}^{(+)}_{\mathcal{S}}$.
Since 
$$\mathcal{S}\left[\mathcal{T}_2\left[\phi\right]\right](z)=\int_0^{\infty} \phi\left(\zeta\right)d \zeta-z g(z),$$
if $b>0$ Eq. (\ref{Eqg2}) has 
one solution for $c> 0$.
The solution is real valued, $z=\xi_2$, and $\left(c/b\right)>\xi_2>0 $. If $b=0$ Eq. (\ref{Eqg2})
has 
one solution for $  \int_0^{\infty} \phi\left(\zeta\right)d \zeta >c>0$.
The solution is real valued and positive. No solution exists, 
otherwise.

Consider the operator $\mathcal{T}_3:\phi\left(\zeta\right)\to 
\phi\left( \zeta\right)/\left(\zeta+a\right)$ and the following
transcendental equation in $\mathbb{C}\backslash \left(-\infty\right. ,0]$,
\begin{eqnarray}
&&\frac{g(z) -g \left(a\right)}{z-a}+ b z + c=0, \label{Eqg3}
\end{eqnarray} where $g \in \mathcal{L}^{(+)}_{\mathcal{S}}$.
Since $$\mathcal{S}\left[\mathcal{T}_3\left[\phi\right]\right](z)
=\frac{g(z)-g(a)}{a-z},$$ if $b>0$ Eq. (\ref{Eqg3})
has 
one solution for $\mu_{\phi}> 0$, where \begin{equation}
\mu_{\phi}=\int_0^{\infty}\frac{\phi\left(\zeta\right)}{\zeta\left(\zeta+a\right)}\, d \zeta-c. \label{muphi}
\end{equation}
The solution is real valued, $z= \xi_3$, and $\left(\mu_{\phi}/b\right)>\xi_3>0 $.
If $b=0$ Eq. (\ref{Eqg3}) has 
one solution for $\mu_{\phi}>c>0$.
The solution is real valued and positive. No solution exists, 
otherwise.
Notice that for $z=a$ the left hand side of Eq. (\ref{Eqg3}) is analytically extended to the value $g^{\prime}(a)+ b a +c$.

Consider the operator $\mathcal{T}_4:\phi\left(\zeta\right)\to \phi\left( a/\zeta\right)/\zeta$ and the following
transcendental equation in $\mathbb{C}\backslash \left(-\infty\right. ,0]$,
\begin{eqnarray}
\frac{g\left(a/z\right)}{z}- b z - c=0, \label{Eqg4}
\end{eqnarray} where $g \in \mathcal{L}^{(+)}_{\mathcal{S}}$.
Since 
$\mathcal{S}\left[\mathcal{T}_4\left[\phi\right]\right](z)=g\left(a/z\right)/z$,
if $b>0$ Eq. (\ref{Eqg4}) has 
one solution for $\nu_{\phi}>0$, where
\begin{equation}
\nu_{\phi}=\int_0^{\infty}\frac{\phi\left(a/\zeta \right)}{\zeta^2}\, d \zeta-c. \label{nuphi}
\end{equation}
The solution is real valued, $z= \xi_4$, and $\left(\nu_{\phi}/b\right)>\xi_4>0$.
If $b=0$ Eq. (\ref{Eqg4})
has 
one solution for $\nu_{\phi}>c>0$. The solution is real valued and positive. No solution exists, 
otherwise.

Consider the operator $\mathcal{T}_5:\phi\left(\zeta\right)\to \phi\left( \zeta^{1/2}\right)$ and the following
transcendental equation in $\mathbb{C}\backslash \left(-\infty\right. ,0]$,
\begin{eqnarray}
g\left( \imath z^{1/2}\right)+g\left( -\imath z^{1/2}\right)- b z - c=0, \label{Eqg5}
\end{eqnarray} where $g \in \mathcal{L}^{(+)}_{\mathcal{S}}$.
Since 
$$\mathcal{S}\left[\mathcal{T}_5\left[\phi\right]\right](z)=g\left( \imath z^{1/2}\right)+g\left( -\imath z^{1/2}\right),$$
 if $b>0$ Eq. (\ref{Eqg5}) 
has 
one solution for $\eta_{\phi}> 0$,
where\begin{equation}
\eta_{\phi}=\int_0^{\infty}\frac{\phi\left(\zeta^{1/2}\right)}{\zeta}\, d \zeta-c. \label{etaphi}
\end{equation}
The solution is real valued, $z=\xi_5$, and $\left(\eta_{\phi}/b\right)>\xi_5>0 $.
 If $b=0$ Eq. (\ref{Eqg5})
has 
one solution for $\eta_{\phi}>c>0$. The solution is real valued and positive. No solution exists, 
otherwise.

\end{corollary}

\begin{proof}
Since $\phi\geq0$ on $\mathbb{R}_+$, the functions $\phi\left(a \zeta\right)$,
$\zeta\phi\left( \zeta\right)$, $\phi\left( \zeta\right)/\left(\zeta+a\right)$, $\phi\left( a/\zeta\right)/\zeta$,
$\phi\left( \zeta^{1/2}\right)/\zeta$ are nonnegative on $\mathbb{R}_+$. The corresponding Stieltjes transforms \cite{EMOT2} provide Eqs. (\ref{Eqg1}), (\ref{Eqg2}), (\ref{Eqg3}), (\ref{Eqg4}), (\ref{Eqg5}). The conditions for the existence, uniqueness, reality and positivity of the solutions are obtained via Theorem \ref{Th1}.
Notice that in each case the solution exists and is positive if the critical value is infinite.
\end{proof}

At this stage we are equipped to analyze a further form of transcendental equations involving the Stieltjes and the following corollary holds.
\begin{corollary}
Consider the following equation in $\mathbb{C}\backslash \left(-\infty\right. ,0]$,
\begin{equation}
\mathcal{S}\left[ \mathcal{T}\left[\phi\right]\right](z)-bz-c=0. \label{Eq} 
\end{equation}
The general operator $\mathcal{T}$ is defined 
as follows,
\begin{eqnarray}
\mathcal{T}=\left(\mathcal{T}_{j_1} \circ \cdots \circ  \mathcal{T}_{j_k}\right). \label{T}
\end{eqnarray}
The index $j_l$ can take the values $0,1,\ldots,5$, for every $l=1,\ldots,k$, and $k\in \mathbb{N}_0$, while $\mathcal{T}_0$ is the identity operator.
The transcendental equation (\ref{Eq}) has in $\mathbb{C}\backslash \left(-\infty\right. ,0]$ 
one solution under each of the two following conditions,
\begin{eqnarray}
&& b>0, \hspace{2em} m_{\mathcal{T}\left[\phi\right]}>0, \label{cond1g}\\
&& b=0, \hspace{2em} m_{\mathcal{T}\left[\phi\right]}>0,\hspace{2em} c>0, \label{cond2g}
\end{eqnarray}
where
\begin{eqnarray}
&&m_{\mathcal{T}\left[\phi\right]}=\int_0^{\infty}\frac{1}{\zeta}\, \mathcal{T} \left[\phi\right]\left(\zeta\right)\, d\zeta-c. \label{mTphi}
\end{eqnarray} 
No solution exists, otherwise. The unique solution is real valued and positive, $z=\kappa_0$, and 
$\left(m_{\mathcal{T}\left[\phi\right]}/b\right)>\kappa_0>0$ if $b>0$.
The condition $m_{\mathcal{T}\left[\phi\right]}=+\infty$ is included. 
Notice that the operator $\mathcal{T}_2$ requires the involved integral to be finite.
\end{corollary}
\begin{proof}
Due to the definition of the operators $\mathcal{T}_0, \ldots, \mathcal{T}_5$, the function $\mathcal{T} \left[\phi\right]\left(\zeta\right)$ is nonegative on $\mathbb{R}_+$.
Consequently, the proof is obtained by directly applying Theorem \ref{Th1} to Eq. (\ref{Eq}). 
\end{proof}

\section{Applications}

We analyze some examples of transcendental equations in $\mathbb{C}\backslash \left(-\infty\right. ,0]$ that contain exponential and hyperbolic forms, power laws, logarithms and special functions. These equations are obtained 
via the Stieltjes transforms of nonnegative real valued functions reported in Ref. \cite{EMOT2,HbkST,GHMC1967,PBMIS2,W0,W1,W2,W3,Hen2,PBMIS3,Mainardi2010} and the same notation is adopted here. The existence, uniqueness, reality, nonnegativity and bounds of the solution are studied via the theorem and corollaries that are introduced above. For the sake of clarity, we remind the constraints on the parameters involved, $a,b>0$, and $c\in\mathbb{R}$.

\subsection{Exponential forms}

The transcendental equation 
\begin{eqnarray}
\pi\,z^{-1/2}\left(1-e^{-2 a z^{1/2}}\right)\Big/2-b z-c=0,\hspace{1em}\left|\arg z\right|<\pi, \label{e1}
\end{eqnarray}
has 
one solution under each of the following conditions, 
\begin{eqnarray}
&& b>0,\hspace{1em} \pi a>c, \nonumber \\
&& b=0, \hspace{1em}\pi a>c>0. \nonumber 
\end{eqnarray}
 The solution is real valued and positive, $z=x_1>0$, and $\left(\pi a -c\right)/b>x_1>0$ if $b>0$. No solution exists, otherwise.
As proof \cite{EMOT2}, apply Theorem \ref{Th1} to $\varphi\left(\zeta\right)=\zeta^{-1/2}\sin^2\left(a \zeta^{1/2}\right)$.

The transcendental equation 
\begin{eqnarray}
\pi z^{-1/2}\left(1+e^{-2 a z^{1/2}}\right)\Big/2-b z-c=0,\hspace{1em}\left|\arg z\right|<\pi, \label{e2}
\end{eqnarray}
has 
one solution if $b>0$, or if $b=0$ and $c>0$.
 The solution is real valued and positive. No solution exists, otherwise. 
As proof \cite{EMOT2}, apply Theorem \ref{Th1} to $\varphi\left(\zeta\right)=\zeta^{-1/2}\cos^2\left(a \zeta^{1/2}\right)$.
Notice that the critical value is infinite.

\subsection{Hyperbolic forms}

The transcendental equation 
\begin{eqnarray}
&&\pi z^{-1/2}\frac{\left(\beta/\lambda-\lambda/\beta\right) \sinh \left(2 a z^{1/2}\right)/2-1}{\left(\beta \sinh \left(a z^{1/2}
\right)\right)^2-\left(\lambda \cosh \left(a z^{1/2}\right)\right)^2}-b z -c=0,\label{hyperbolic1}\\
&&\beta \lambda>0,\hspace{1em}\left|\arg z\right|<\pi, \nonumber
\end{eqnarray}
has one solution if $b>0$, or if $b=0$ and $c>0$.
 The solution is real valued and positive. No solution exists, otherwise. 
As proof \cite{EMOT2}, apply Theorem \ref{Th1} to $\varphi\left(\zeta\right)=\zeta^{-1/2}\Big/\Big(\left(\beta\sin\left( a \zeta^{1/2}\right)\right)^2+
\left(\lambda\cos\left( a \zeta^{1/2}\right)\right)^2\Big)$.
Notice that the critical value is infinite.

\subsection{Power laws}
The transcendental equation \begin{eqnarray}
b z- \pi \csc \left(\pi \alpha\right) z^{\alpha-1}+c =0,  \hspace{1em} 1>\alpha>0,\hspace{1em}\left|\arg z\right|<\pi,   \label{p1}
\end{eqnarray}
has one solution if $b>0$, or if $b=0$ and $c>0$. The solution is real valued and positive. 
No solution exists, otherwise. 
As proof \cite{EMOT2}, apply Theorem \ref{Th1} to $\varphi\left(\zeta \right)=\zeta^{\alpha-1}$. Notice that the critical value is infinite.

The transcendental equation 
\begin{eqnarray}
&&\hspace{-4em}\pi \csc\left(\pi \alpha\right)\, \frac{z^{\alpha}-a^{\alpha} }{z-a}-bz-c=0, 
\hspace{1em}1>\alpha>-1,\hspace{1em}\left|\arg z\right|<\pi, \label{p2}
\end{eqnarray}
has 
one solution under each of the following conditions, 
\begin{eqnarray}
&& b>0,\hspace{1em} 1>\alpha>0, \hspace{1em} c<\pi a^{\alpha-1} \csc \left(\pi \alpha\right), \nonumber \\
&&b>0, \hspace{1em}0\geq \alpha>-1, \nonumber\\
&&b=0,\hspace{1em} 1>\alpha>0, \hspace{1em}\pi a^{\alpha-1}\csc\left(\pi \alpha\right)>c>0, \nonumber \\
&& b=0, \hspace{1em}0\geq \alpha>-1, \hspace{1em}c>0. \nonumber 
\end{eqnarray}
 The solution is real valued and positive, $z=x_2$, and $\left(\pi a^{\alpha-1} \csc \left(\pi \alpha\right)-c\right)/b>x_2>0$ if
$b>0$ and $1>\alpha>0$. No solution exists, otherwise. 
The left hand side of Eq. (\ref{p2}) is analytically extended in $z=a$ to the 
value $\left(\pi \alpha a^{\alpha-1} \csc\left(\pi \alpha\right)-b a-c \right)$.
As proof \cite{EMOT2}, apply Theorem \ref{Th1} to $\varphi\left(\zeta\right)=\zeta^{\alpha}/\left(\zeta+a\right)$.

The transcendental equation 
\begin{eqnarray}
&&\hspace{-4em}\pi\,\frac{ a^{\alpha-1} \sec \left(\pi\alpha/2\right)z-2 \csc \left(\pi \alpha\right)z^{\alpha}
+a^{\alpha} \csc \left(\pi\alpha/2\right) }{2\left(z^2+a^2\right)}-b z-c=0,  \label{p3}\\
&&\hspace{-4em}2>\alpha>-1, \hspace{1em}\left|\arg z\right|<\pi,\nonumber
\end{eqnarray}
has 
one solution under each of the following conditions,
\begin{eqnarray}
&& b>0,\hspace{1em} 2>\alpha>0, \hspace{1em} c<\pi a^{\alpha-2} \csc \left(\pi \alpha/2\right)/2, \nonumber \\
&&b>0, \hspace{1em}0\geq \alpha>-1, \nonumber\\
&& b=0, \hspace{1em}2> \alpha>0, \hspace{1em}\pi a^{\alpha-2} \csc \left(\pi \alpha/2\right)/2>c>0, \nonumber \\
&&b=0,\hspace{1em} 0\geq \alpha>-1,\hspace{1em} c>0. \nonumber 
\end{eqnarray}
 The solution is real valued and positive, $z=x_3$, and 
$$\hspace{-4em}\left(\pi a^{\alpha-2} \csc \left(\pi \alpha/2\right)/2-c\right)/b>x_3>0, \hspace{1em}
b>0, \hspace{1em}2>\alpha>0.$$ No solution exists, otherwise. 
As proof \cite{EMOT2}, apply Theorem \ref{Th1} to $\varphi\left(\zeta\right)=\zeta^{\alpha}/\left(\zeta^2+a^2\right)$.

\subsection{Logarithms}

The transcendental equation 
\begin{equation}
\frac{\log\left(a/z\right)}{\left(z-a\right)}+bz+c=0, \hspace{1em}\left|\arg z\right|<\pi, \label{l1}
\end{equation}
has 
one solution if $b>0$, or if $b=0$ and $c>0$.
 The solution is real valued and positive. No solution exists, otherwise.
The left hand side of Eq. (\ref{p2}) is analytically extended in $z=a$ to the 
value $\left(b a +c-1/a\right)$ .
As proof \cite{EMOT2}, apply Theorem \ref{Th1} to $\varphi\left(\zeta\right)=1/\left(\zeta+a\right)$.
Notice that the critical value is infinite.

The transcendental equation 
\begin{equation}
\frac{ \log\left(z/a\right)- \pi z /(2 a)}{\left(z^2+a^2\right)}+bz+c=0, \hspace{1em}\left|\arg z\right|<\pi, \label{l2}
\end{equation}
has 
one solution if $b>0$, or if $b=0$ and $c>0$.
 The solution is real valued and positive. No solution exists, otherwise. 
As proof \cite{EMOT2}, apply Theorem \ref{Th1} to $\varphi\left(\zeta\right)=1/\left(\zeta^2+a^2\right)$.
Notice that the critical value is infinite.

The transcendental equation 
\begin{equation}
\frac{z \log\left(z/a\right)+ \pi a /2}{\left(z^2+a^2\right)}-bz-c=0, \hspace{1em}\left|\arg z\right|<\pi, \label{l3}
\end{equation}
has 
one solution under each of the following conditions,
\begin{eqnarray}
&& b>0,\hspace{1em} c<\pi / \left(2 a \right), \nonumber \\
&&b=0, \hspace{1em}\pi / \left(2 a \right)>c>0. \nonumber
\end{eqnarray}
 The solution is real valued and positive, $z=x_4$, and $\left(\pi / \left(2 a \right)-c\right)/b>x_4>0$ if
$b>0$. No solution exists, otherwise. 
As proof \cite{EMOT2}, apply Theorem \ref{Th1} to $\varphi\left(\zeta\right)=\zeta/\left(\zeta^2+a^2\right)$.

The transcendental equation \begin{eqnarray}
2 \pi z^{-1/2}\log \left(a^{1/2} z^{1/2} + 1\right)-b z-c=0,\hspace{1em}\left|\arg z\right|<\pi,\label{l4}
\end{eqnarray} has 
one solution under each of the following conditions,
\begin{eqnarray}
&& b>0,\hspace{1em} c< 2\pi a^{1/2}, \nonumber \\
&&b=0, \hspace{1em}2\pi a^{1/2}>c>0. \nonumber
\end{eqnarray}
 The solution is real valued and positive, $z=x_5$, and $\left(2\pi a^{1/2}-c\right)/b>x_5>0$ if
$b>0$. No solution exists, otherwise. 
As proof \cite{EMOT2}, apply Theorem \ref{Th1} to $\varphi\left(\zeta\right)=\zeta^{-1/2} \log \left(a \zeta+1\right)$.

The transcendental equation 
\begin{equation}
\frac{\pi^2 +\log^2\left(z/a\right)}{2\left(z+a\right)}-bz-c=0, \hspace{1em}\left|\arg z\right|<\pi, \label{l5}
\end{equation}
has 
one solution if $b>0$, or if $b=0$ and $c>0$.
 The solution is real valued and positive. No solution exists, otherwise. 
As proof \cite{EMOT2}, apply Theorem \ref{Th1} to $\varphi\left(\zeta\right)=\log \left(\zeta/a\right)/\left(\zeta-a\right)$ for $\zeta\neq a $
and $\varphi\left(a\right)=1/a$. Notice that the critical value is infinite.

The transcendental equation \begin{eqnarray}
&&\hspace{-4em}\pi\,\frac{a^{\alpha} \log\left(a/z\right)/\pi+ \csc \left(\pi \alpha\right)z^{\alpha
} -a^{\alpha} \cot\left(\pi \alpha\right)}{a+z}-b z-c=0, \label{l6}\\
&&\hspace{-4em} 1>\alpha>0, \hspace{1em}\left|\arg z\right|<\pi, \nonumber
\end{eqnarray}
has one solution if $b>0$, or if $b=0$ and $c>0$. The solution is real valued and positive. 
As proof \cite{EMOT2}, apply Theorem \ref{Th1} to $\varphi\left(\zeta\right)=\left(\zeta^{\alpha}-a^{\alpha}\right)/\left(\zeta-a\right)$ for $\zeta\neq a$
and $\varphi(a)=\alpha a^{\alpha-1}$. Notice that the critical value is infinite.

The transcendental equation \begin{eqnarray}
&& \hspace{-4em}\frac{\pi}{z+a}\Bigg( \csc\left(\pi \alpha\right)\frac{\pi \cot\left(\pi \alpha\right)\left(z^{\alpha}-a^{\alpha}\right)-z^{\alpha}
\log \left(z/a\right)}{z-a}+\frac{\pi a^{\alpha-1}}{4 \cos^2\left(\pi \alpha/2\right)}\Bigg)\nonumber \\
&&\hspace{-4em}-bz -c=0,\hspace{1em}1>\alpha>-1, 
\hspace{1em}\left|\arg z\right|<\pi, \label{l7} 
\end{eqnarray}has 
one solution
under each of the following conditions,
\begin{eqnarray}
&& b>0,\hspace{1em} 1>\alpha>0, \hspace{1em} c<\pi^2 a^{\alpha-2}\csc^2 \left(\pi \alpha/2\right)/4, \nonumber \\
&&b>0, \hspace{1em}0\geq \alpha>-1, \nonumber\\
&& b=0, \hspace{1em}1> \alpha>0, \hspace{1em}\pi^2 a^{\alpha-2}\csc^2 \left(\pi \alpha/2\right)/4>c>0, \nonumber\\
&&b=0,\hspace{1em} 0\geq \alpha>-1,\hspace{1em} c>0. \nonumber 
\end{eqnarray}
 The solution is real valued and positive, $z=x_6$, and 
$$\hspace{-2em}\left(\pi^2 a^{\alpha-2}\csc^2 \left(\pi \alpha/2\right)/4-c\right)/b>x_6>0, \hspace{1em}b>0, \hspace{1em}1>\alpha>0.$$ No solution exists, otherwise. The left hand side of Eq. (\ref{l5}) is analytically extended in $z=a$
to the value $$ \hspace{-3em}\frac{\pi a^{\alpha-2}}{2}\Bigg( \csc\left(\pi \alpha\right)\left(\pi \alpha \cot\left(\pi \alpha\right)-1\right)+
\frac{\pi}{4 \cos^2\left(\pi \alpha/2\right)}\Bigg)-ba -c.$$
 As proof \cite{EMOT2}, apply Theorem \ref{Th1} to $\varphi\left(\zeta\right)=\zeta^{\alpha} \log \left(\zeta/a\right)/\left(\zeta^2-a^2\right)$ for $\zeta\neq a$ and $\varphi(a)=a^{\alpha-2}/2$.
Notice that the critical value is infinite for $0\geq \alpha>-1$.

\subsection{Special functions}

We consider transcendental equations involving various types of special functions. In some cases the analysis of the solution provides additional information about the zeros of these functions.

\subsubsection{Incomplete Gamma function}
The transcendental equation \begin{eqnarray}
\hspace{-2em}\Gamma\left(1-\alpha\right) z^{-\alpha} e^{a z} \Gamma\left(\alpha, a z\right) -bz-c=0,
\hspace{1em}\alpha<1, \hspace{1em} \left|\arg z\right|<\pi,\label{gamma1}
\end{eqnarray}
in terms of the incomplete Gamma function \cite{EMOT2}, has 
one solution
under each of the following conditions,
\begin{eqnarray}
&& b>0,\hspace{1em} 1>\alpha \geq 0, \nonumber \\
&& b>0, \hspace{1em} \alpha<0, \hspace{1em} c<a^{\alpha} \Gamma\left(-\alpha\right), \nonumber \\
&& b=0, \hspace{1em} 1>\alpha \geq 0, \hspace{1em} c>0,\nonumber \\
&& b=0, \hspace{1em} \alpha < 0, \hspace{1em} 0<c<a^{\alpha} \Gamma\left(-\alpha\right). \nonumber 
\end{eqnarray}
 The solution is real valued and positive, $z=x_7$, and 
$$\hspace{1em}\left(a^{\alpha} \Gamma\left(-\alpha\right)-c\right)/b>x_7>0,\hspace{1em} b>0,\hspace{1em}
\alpha< 0.$$
 No solution exists, otherwise. 
The nonexistence of any solution for $b=c=0$ suggests that the function
$\Gamma\left(\alpha, z\right)$ has no zero for $\alpha<1$, $\left|\arg z\right|<\pi$ and $z\neq 0$. As proof \cite{EMOT2}, apply Theorem \ref{Th1} to $\varphi\left(\zeta\right)=\zeta^{-\alpha} \exp\left\{-a \zeta\right\}$.
Notice that the critical value is infinite for $0\leq \alpha<1$.

The transcendental equation \begin{eqnarray}
\hspace{-3em}\Gamma\left(1-\alpha\right) z^{\alpha-1} e^{a/z } \Gamma\left(\alpha, a /z\right) -bz-c=0,
\hspace{1em}\alpha<1, \hspace{1em} \left|\arg z\right|<\pi,\label{gamma2}
\end{eqnarray}
in terms of the incomplete Gamma function \cite{EMOT2}, has 
one solution
under each of the following conditions,
\begin{eqnarray}
&& b>0,\hspace{1em} c<a^{\alpha-1} \Gamma\left(1-\alpha\right), \nonumber \\
&& b=0, \hspace{1em}0<c<a^{\alpha-1} \Gamma\left(1-\alpha\right), \nonumber
\end{eqnarray}
 The solution is real valued and positive, $z=x_8$, and 
$$\left(a^{\alpha-1} \Gamma\left(1-\alpha\right)-c\right)/b>x_8>0, \hspace{1em}  b>0.$$
 No solution exists, otherwise. 
 As proof \cite{EMOT2}, apply Theorem \ref{Th1} to $\varphi\left(\zeta\right)=\zeta^{\alpha-1} \exp\left\{-a/ \zeta\right\}$.

The transcendental equation \begin{eqnarray}
&&\hspace{-6em}\Gamma\left(2\alpha +1\right) z^{\alpha}\Bigg(e^{\imath \left(\pi \alpha+a z^{1/2}\right)}
\Gamma\left(-2\alpha, \imath a z^{1/2} \right)+
e^{-\imath \left(\pi \alpha+a z^{1/2}\right)}\nonumber \\&&\hspace{-6em}\times\,
\Gamma\left(-2\alpha, -\imath a z^{1/2} \right)\Bigg) -b z -c =0,
\hspace{1em}\alpha>-1, \hspace{1em} \left|\arg z\right|<\pi,\label{gamma3}
\end{eqnarray}
in terms of the incomplete Gamma function \cite{EMOT2}, has 
one solution
under each of the following conditions,
\begin{eqnarray}
&& b>0,\hspace{1em} \alpha>0, \hspace{1em} c<2 a^{-2\alpha} \Gamma\left(2\alpha\right), \nonumber \\
&& b>0, \hspace{1em} 0\geq \alpha>-1,  \nonumber \\
&& b=0, \hspace{1em}\alpha>0, \hspace{1em} 0<c<2 a^{-2\alpha} \Gamma\left(2\alpha\right), \nonumber \\
&& b=0, \hspace{1em}  0\geq \alpha>-1,\hspace{1em} c>0. \nonumber 
\end{eqnarray}
 The solution is real valued and positive, $z=x_9$, and 
$$\left(2 a^{-2\alpha} \Gamma\left(2\alpha\right)-c\right)/b>x_9>0, \hspace{1em}b>0, \hspace{1em}\alpha>0.$$
 No solution exists, otherwise. 
 As proof \cite{EMOT2}, apply Theorem \ref{Th1} to $\varphi\left(\zeta\right)=\zeta^{\alpha} \exp\left\{-a  \zeta^{1/2}\right\}$.
Notice that the critical value is infinite for $0\geq \alpha>-1$.

\subsubsection{Logaritmic derivative of the Gamma functions}

The transcendental equation 
\begin{eqnarray}
\hspace{-2em}\psi \left(a z^{1/2}\right)-\log\left(a z^{1/2}\right) + \left(2 a z^{1/2}\right)^{-1}+ b z +c=0, \hspace{1em}
\left|\arg z\right|<\pi, \label{psiG1}
\end{eqnarray}
in terms of the logarithmic derivative of the Gamma function \cite{EMOT2}, $\psi(z)=\left(\log\left(\Gamma (z)\right)\right)^{\prime}$, has 
one solution if $b>0$, or if $b=0$ and $c>0$.
 The solution is real valued and positive. No solution exists, otherwise.
As proof \cite{EMOT2}, apply Theorem \ref{Th1} to $\varphi\left(\zeta\right)=1\big/\left(\exp\left\{a \zeta^{1/2}\right\}-1\right)$.
Notice that the critical value is infinite.

The transcendental equation 
\begin{eqnarray}
\psi \left( z^{1/2}/2+1/2\right)-\psi \left( z^{1/2}/2\right)-z^{-1/2}-b z - c =0, \hspace{1em}
\left|\arg z\right|<\pi, \label{psiG2}
\end{eqnarray}
in terms of the logarithmic derivative of the Gamma function \cite{EMOT2}, has 
one solution if $b>0$, or if $b=0$ and $c>0$.
 The solution is real valued and positive. No solution exists, otherwise. 
As proof \cite{EMOT2}, apply Theorem \ref{Th1} to $\varphi\left(\zeta\right)=1/\sinh\left( \pi \zeta^{1/2}\right)$. Notice that the critical value is infinite.

The transcendental equation 
\begin{eqnarray}
\hspace{-1.1em}z^{-1/2}\Bigg(\psi \left( \frac{z^{1/2}}{2}+\frac{3}{4}\right)-\psi \left( \frac{z^{1/2}}{2}+\frac{1}{4}\right)\Bigg)-bz-c=0,\hspace{1em}
\left|\arg z\right|<\pi, \label{psiG3}
\end{eqnarray}
in terms of the logarithmic derivative of the Gamma function \cite{EMOT2}, has 
one solution if $b>0$, or if $b=0$ and $c>0$.
 The solution is real valued and positive. No solution exists, otherwise. 
As proof \cite{EMOT2}, apply Theorem \ref{Th1} to $\varphi\left(\zeta\right)=1\big/\left( \zeta^{1/2}\cosh\left( \pi \zeta^{1/2}\right)\right)$.
Notice that the critical value is infinite.

\subsubsection{Gauss hypergeometric series}

The transcendental equation 
\begin{eqnarray}
&&\hspace{-6em}\frac{\Gamma\left(\alpha\right) \Gamma\left(\beta-\alpha\right)}{a^{\beta-1}\Gamma\left(\beta\right)}z^{\alpha-1}\,
_2F_1\left(\beta-1,\alpha;\beta;1-z/a\right)-bz-c=0,\label{hyp1}\\&&
\hspace{-6em}\beta>\alpha>0, \hspace{1em}
\left|\arg z\right|<\pi, \nonumber
 \end{eqnarray} in terms of the Gauss hypergeometric series \cite{ EMOT2},  
has one solution
under each of the following conditions,
\begin{eqnarray}
&& b>0,\hspace{1em} \beta>\alpha>1, \hspace{1em} c<a^{\alpha-\beta}\Gamma\left(\alpha-1\right)
\Gamma\left(\beta-\alpha\right)/\Gamma\left(\beta-1\right),\nonumber \\
&&b>0, \hspace{1em}1\geq \alpha>0, \nonumber\\
&& b=0, \hspace{1em}\beta>\alpha>1, \hspace{1em}a^{\alpha-\beta}\Gamma\left(\alpha-1\right)
\Gamma\left(\beta-\alpha\right)/\Gamma\left(\beta-1\right)>c>0, \nonumber\\
&&b=0,\hspace{1em} 1\geq\alpha>0,\hspace{1em} c>0. \nonumber 
\end{eqnarray}
 The solution is real valued and positive, $z=x_{10}$, and $$\frac{a^{\alpha-\beta}\Gamma\left(\alpha-1\right)
\Gamma\left(\beta-\alpha\right)/\Gamma\left(\beta-1\right)-c}{b}>x_{10}>0, \hspace{1em}b>0,\hspace{1em}\beta>\alpha>1.$$ 
No solution exists, otherwise. 
The nonexistence of any solution for $b=c=0$ suggests that the function 
$_2F_1\left(\beta-1,\alpha;\beta;1-z/a\right)$ has no zero for
$\beta>\alpha>0$, $\left|\arg z\right|<\pi$ and $z\neq0$. As proof \cite{EMOT2}, apply Theorem \ref{Th1} to $\varphi\left(\zeta\right)=\zeta^{\alpha-1} /\left(\zeta +a\right)^{\beta-1}$.
Notice that the critical value is infinite for $1\geq \alpha>0$.

\subsubsection{Incomplete Beta function}

The function $I_{z}\left(\beta,\alpha\right)$ is defined in terms of the incomplete Beta function $B_{z}\left(\beta,\alpha\right)$ as follows \cite{EMOT2},\\
$I_{z}\left(\beta,\alpha\right)=B_{z}\left(\beta,\alpha\right)/B_{1}\left(\beta,\alpha\right)$.
The transcendental equation 
\begin{eqnarray}
&&\hspace{-6em}\pi \csc\left(\pi \alpha\right) z^{-\alpha} \left(a-z\right)^{-\beta} I_{1-z/a}\left(\beta,\alpha\right)
-bz-c=0,\label{IB1}\\
 &&\hspace{-6em} 1>\alpha>-\beta, \hspace{1em}\left|\arg z\right|<\pi,\nonumber
\end{eqnarray} has 
one solution
under each of the following conditions,
\begin{eqnarray}
&& b>0,\hspace{1em} 0>\alpha>-\beta, \hspace{1em} c<a^{-\alpha-\beta}\Gamma\left(-\alpha\right)
\Gamma\left(\alpha+\beta\right)/\Gamma\left(\beta\right), \nonumber \\
&&b>0, \hspace{1em}1> \alpha\geq 0, \nonumber\\
&& b=0, \hspace{1em}0>\alpha>-\beta, \hspace{1em} a^{-\alpha-\beta}\Gamma\left(-\alpha\right)
\Gamma\left(\alpha+\beta\right)/\Gamma\left(\beta\right)>c>0, \nonumber\\
&&b=0,\hspace{1em} 1> \alpha\geq 0,\hspace{1em} c>0. \nonumber 
\end{eqnarray}
 The solution is real valued and positive, $z=x_{11}$, and $$\left(a^{-\alpha-\beta}\Gamma\left(-\alpha\right)
\Gamma\left(\alpha+\beta\right)/\Gamma\left(\beta\right)-c\right)\Big/b>x_{11}>0, \hspace{1em}b>0,  \hspace{1em} 0>\alpha>-\beta.
$$ No solution exists, otherwise. The nonexistence of any solution for $b=c=0$ suggests that the function $I_{1-z/a}\left(\beta,\alpha\right)$
has no zero for $1>\alpha>-\beta$, $\left|\arg z\right|<\pi$, $z\neq a,0$.
 As proof \cite{EMOT2}, apply Theorem \ref{Th1} to $\varphi\left(\zeta\right)=\zeta^{-\alpha} /\left(\zeta +a\right)^{\beta}$.
Notice that the critical value is infinite for $1>\alpha\leq 0$.

\subsubsection{Exponential integral}

The transcendental equation 
\begin{equation}
e^{\alpha z} {\rm Ei} \left(-\alpha z\right)+ b z +c=0, \hspace{1em} \alpha>0, \hspace{1em}
\left|\arg z\right|<\pi, \label{Ei1}
\end{equation}
in terms of the exponential integral function \cite{EMOT2}, has 
one solution if $b>0$, or if $b=0$ and $c>0$.
 The solution is real valued and positive. No solution exists, otherwise. 
The nonexistence of any solution for $b=c=0$ suggests that the exponential integral ${\rm Ei}(-z)$ has no zero for $\left|\arg z\right|<\pi$
and $z\neq 0$.
As proof \cite{EMOT2}, apply Theorem \ref{Th1} to $\varphi\left(\zeta\right)=\exp\left\{-\alpha \zeta\right\}$.
Notice that the critical value is infinite.

The transcendental equation 
\begin{eqnarray}
&&\hspace{-6em}e^{\alpha z} \left({\rm Ei} \left(-\alpha z-\alpha d\right)-{\rm Ei} \left(-\alpha z\right) \right)  - b z -c=0, \label{Ei2}\\
&&\hspace{-6em} \alpha>0,\hspace{1em} d>0,\hspace{1em}
\left|\arg z\right|<\pi, \nonumber
\end{eqnarray}
in terms of the exponential integral function \cite{EMOT2}, has 
one solution if $b>0$, or if $b=0$ and $c>0$.
 The solution is real valued and positive. No solution exists, otherwise. 
As proof \cite{EMOT2}, apply Theorem \ref{Th1} to $\varphi\left(\zeta\right)=\exp\left\{-\alpha \zeta\right\}$ for $0<\zeta<d$ and $\varphi\left(\zeta\right)=0$ for $\zeta>d$.
Notice that the critical value is infinite.

The transcendental equation \begin{eqnarray}
&&\hspace{-6em} e^{\alpha z} {\rm Ei} \left(-\alpha z-\alpha d\right)
 + b z +c=0, \label{Ei3}\\
&&\hspace{-6em} \alpha>0,\hspace{1em} d>0, \hspace{1em}
\left|\arg z\right|<\pi, \nonumber
\end{eqnarray}
in terms of the exponential integral function \cite{EMOT2}, has 
one solution under each of the following conditions,
\begin{eqnarray}
&& b>0,\hspace{1em} c< - {\rm Ei}\left(-\alpha d\right), \nonumber \\
&& b=0, \hspace{1em} 0<c<-{\rm Ei}\left(-\alpha d\right), \nonumber
\end{eqnarray}
 The solution is real valued and positive, $z=x_{12}$, and 
$$\left(-{\rm Ei}\left(-\alpha d\right)-c\right)/b>x_{12}>0, \hspace{1em}b>0.$$ No solution exists, otherwise. 
 As proof \cite{EMOT2}, apply Theorem \ref{Th1} to $\varphi\left(\zeta\right)=0$ for $0<\zeta<d$ and $\varphi\left(\zeta\right)=\exp\left\{-\alpha \zeta\right\}$ for $\zeta>d$.

The transcendental equation \begin{eqnarray}
&&\hspace{-5em}(-1)^{n+1} z^n e^{\alpha z} {\rm Ei}\left(-\alpha z\right) 
+ \sum_{j=1}^n (-1)^{n-j} (j-1)!\, \alpha ^{-j} z^{n-j}-bz-c=0,\label{Ei4}\\
&&\hspace{-5em}\alpha>0, \hspace{1em}n \in \mathbb{N}_0, \hspace{1em}\left|\arg z\right|<\pi, \nonumber
\end{eqnarray}in terms of the exponential integral function \cite{EMOT2}, has 
one solution
under each of the following conditions,
\begin{eqnarray}
&& b>0,\hspace{1em} c< (n-1)!\, \alpha^{-n}, \nonumber \\
&& b=0, \hspace{1em} 0<c<(n-1)!\, \alpha^{-n}. \nonumber 
\end{eqnarray}
 The solution is real valued and positive, $z=x_{13}$, and 
$$\left((n-1)! \,\alpha^{-n}-c\right)/b>x_{13}>0,  \hspace{1em}b>0.$$ No solution exists, otherwise. 
 As proof \cite{EMOT2}, apply Theorem \ref{Th1} to $\varphi\left(\zeta\right)=\zeta^n \exp \left\{-\alpha \zeta\right\}$.

The transcendental equation 
\begin{eqnarray}
\hspace{-2em}\left(\log\left(\alpha \gamma z\right)- e^{\alpha z}{\rm Ei} \left(-\alpha z\right)\right)/z-bz-c=0,
 \hspace{1em} \alpha>0,\hspace{1em}
\left|\arg z\right|<\pi, \label{Ei5}
\end{eqnarray}
in terms of the exponential integral function and the Euler-Mascheroni constant $\gamma=\exp\left\{C\right\}$ \cite{EMOT2}, has 
one solution if $b>0$, or if $b=0$ and $c>0$.
 The solution is real valued and positive. No solution exists, otherwise. 
As proof \cite{EMOT2}, apply Theorem \ref{Th1} to $\varphi\left(\zeta\right)=\left(1-\exp\left\{-\alpha \zeta\right\}\right)/
\zeta$.
Notice that the critical value is infinite.

\subsubsection{Error function}

The transcendental equation 
\begin{eqnarray}
\pi z^{-1/2}e^{a z} {\rm Erfc}\left(a^{1/2} z^{1/2}\right)-b z-c=0, \hspace{1em}
\left|\arg z\right|<\pi, \label{Erfc1}
\end{eqnarray}
in terms of the error function \cite{EMOT2}, has 
one solution if $b>0$, or if $b=0$ and $c>0$.
 The solution is real valued and positive. No solution exists, otherwise. 
The nonexistence of any solution for $b=c=0$ suggests that the error function ${\rm Erfc}(z)$ has no zero for $\left|\arg z\right|<\pi/2$ and $z\neq 0$.
As proof \cite{EMOT2}, apply Theorem \ref{Th1} to $\varphi\left(\zeta\right)=\zeta^{-1/2}
\exp\left\{- a \zeta\right\}$.
Notice that the critical value is infinite.

The transcendental equation 
\begin{eqnarray}
\pi z^{1/2} e^{a z}{\rm Erfc}\left(a^{1/2} z^{1/2}\right) + b z+c- \left(\pi /a\right)^{1/2}=0,\hspace{1em}
\left|\arg z\right|<\pi, \label{Erfc2}
\end{eqnarray}
in terms of the error function \cite{EMOT2}, has 
one solution 
under each of the following conditions,
\begin{eqnarray}
&& b>0,\hspace{1em} c< \left(\pi/a\right)^{1/2}, \nonumber \\
&& b=0, \hspace{1em} 0<c<\left(\pi/a\right)^{1/2}. \nonumber 
\end{eqnarray}
 The solution is real valued and positive, $z=x_{14}$, and $$\left(\left(\pi/a\right)^{1/2}-c\right)\Big/b
>x_{14}>0.$$ No solution exists, otherwise. 
As proof \cite{EMOT2}, apply Theorem \ref{Th1} to $\varphi\left(\zeta\right)=\zeta^{1/2}
\exp\left\{- a \zeta\right\}$.

\subsubsection{Cosine and sine integrals}

The transcendental equation 
\begin{eqnarray}
&&\hspace{-2em}2 \cos \left(a z^{1/2}\right) {\rm ci}\left(a z^{1/2}\right)- 2 \sin \left(a z^{1/2}\right) {\rm si}\left(a z^{1/2}\right)\nonumber \\ &&\hspace{-2em}
-b z-c =0,\hspace{1em}
\left|\arg z\right|<\pi, \label{ci1}
\end{eqnarray}
in terms of the cosine and sine integral functions \cite{EMOT2}, has 
one solution if $b>0$, or if $b=0$ and $c>0$.
 The solution is real valued and positive. No solution exists, otherwise. 
As proof \cite{EMOT2}, apply Theorem \ref{Th1} to $\varphi\left(\zeta\right)=\exp\left\{- a \zeta^{-1/2}\right\}$.
Notice that the critical value is infinite.

The transcendental equation 
\begin{eqnarray}
&&\hspace{-2em}2 z^{-1/2}\Big(\sin \left(a z^{1/2}\right) {\rm ci}\left(a z^{1/2}\right)+  \cos \left(a z^{1/2}\right) {\rm si}\left(a z^{1/2}\right)\Big) \nonumber \\
&&\hspace{-2em}+ b z+c =0,\hspace{1em}
\left|\arg z\right|<\pi, \label{ci2}
\end{eqnarray}
in terms of the cosine and sine integral functions \cite{EMOT2}, has 
one solution if $b>0$, or if $b=0$ and $c>0$.
 The solution is real valued and positive. No solution exists, otherwise. 
As proof \cite{EMOT2}, apply Theorem \ref{Th1} to $\varphi\left(\zeta\right)=\zeta^{-1/2}\exp\left\{- a \zeta^{-1/2}\right\}$.
Notice that the critical value is infinite.

\subsubsection{Whittaker function}

The transcendental equation \begin{eqnarray}
\hspace{-2em}z^{\left(\alpha+\beta\right)/2-1}e^{z/2} W_{-\left(\alpha+\beta\right)/2,
\left(\beta-\alpha\right)/2}(z)
-bz-c=0,
 \hspace{1em} \left|\arg z\right|<\pi,\label{EqWhittaker}
\end{eqnarray}
in terms of the Whittaker function \cite{SFnist}, where 
\begin{eqnarray}
&&\hspace{-5em} \alpha>-1/2, \hspace{1em} 1/2> \beta>-1/2, \hspace{1em}\text{or}\hspace{1em} \beta>-1/2, \hspace{1em} 1/2> \alpha>-1/2,\label{condWhit}
\end{eqnarray}
has 
one solution if $b>0$, or if $b=0$ and $c>0$.
 The solution is real valued and positive. No solution exists, otherwise. 
 The nonexistence of any solution for $b=c=0$ suggests that under the condition (\ref{condWhit}) the Whittaker function $W_{-\left(\alpha+\beta\right)/2,
\left(\beta-\alpha\right)/2}(z)$ has no zero for $\left|\arg z\right|<\pi$ and $z\neq 0$.
As proof \cite{GHMC1967,HbkST}, apply Theorem \ref{Th1} to   
$$\varphi\left(\zeta\right)=
\frac{\zeta^{\left(\alpha+\beta\right)/2-1}e^{-\zeta/2}}{\Gamma\left(\alpha+1/2\right)\Gamma\left(\beta+1/2\right)}\,
W_{\left(\alpha+\beta\right)/2,\left(\alpha-\beta\right)/2}\left(\zeta\right).$$
The above function is positive on $\mathbb{R}_+$ under the condition (\ref{condWhit}).
 The asymptotic behaviours shown by the Whittaker function under the condition (\ref{condWhit}) induce the critical value to be infinite \cite{SFnist,PBMIS3},
\begin{eqnarray}\int_0^{\infty}\frac{\varphi\left(\zeta\right)}{\zeta}\, d\zeta=
\left.\frac{-\zeta^{\left(\alpha+\beta\right)/2-1}e^{-\zeta/2}\,
W_{\left(\alpha+\beta\right)/2-1,\left(\alpha-\beta\right)/2}\left(\zeta\right)}{\Gamma\left(\alpha+1/2\right)\Gamma\left(\beta+1/2\right)}
\right|^{\infty}_{0}=\infty. 
\end{eqnarray}

\subsubsection{Bessel functions}

The transcendental equation 
\begin{eqnarray}
&&\hspace{-2em}\frac{\pi}{2} J_{\alpha}\left(a z\right)Y_{\alpha}\left(a z\right)
+\frac{a z \Gamma\left(\alpha-1/2\right)}{\pi\Gamma\left(\alpha+3/2\right)} \,
_2F_3\left(1,1;\frac{3}{2},\alpha+\frac{3}{2},\frac{3}{2}-\alpha;-a^2z^2\right)\nonumber \\&&\hspace{-2em}
+\frac{\pi \left(a z\right)^{2\alpha} \tan\left(\pi \alpha\right)}{2^{2\alpha+1}\Gamma^2\left(\alpha+1\right)} \,_1F_2\left(\alpha+\frac{1}{2};
\alpha+1, 2\alpha+1;-a^2z^2\right)\nonumber \\&&\hspace{-2em} +b z +c=0,\hspace{1em}\alpha \in \mathbb{R},\hspace{1em}
\left|\arg z\right|<\pi, \label{EqmB1}
\end{eqnarray}
in terms of the Bessel functions and the generalized hypergeometric functions \cite{EMOT2,SFnist}, has one solution 
under each of the following conditions,
\begin{eqnarray}
&& b>0,\hspace{1em} \alpha>0,\hspace{1em} c< 1/\left(2\alpha\right), \nonumber \\
&& b>0,\hspace{1em} \alpha\leq 0,\nonumber \\
&& b=0,\hspace{1em} \alpha>0,\hspace{1em}1/\left(2\alpha\right)>c>0, \nonumber\\ 
&& b=0,\hspace{1em} \alpha\leq 0,\hspace{1em}c>0. \nonumber 
\end{eqnarray}
 The solution is real valued and positive, $z=x_{15}$, and 
$$\left(1/\left(2\alpha\right)-c\right)/b
>x_{15}>0,\hspace{1em}b>0,\hspace{1em}\alpha>0.$$ No solution exists, otherwise. 
As proof \cite{PBMIS2}, apply Theorem \ref{Th1} to $\varphi\left(\zeta\right)=J^2_{\alpha}\left(\zeta\right)$.

The transcendental equation 
\begin{eqnarray}
\left(2 \pi z\right)^{-1/2}e^{z} K_{\beta}(z)-bz-c=0,\hspace{1em} 1/2 >\beta>-1/2,\hspace{1em} \left|\arg z\right|<\pi,  \label{EqmB2}
\end{eqnarray}
in terms of the modified Bessel function \cite{EMOT2,SFnist}, has 
one solution if $b>0$, or if $b=0$ and $c>0$.
 The solution is real valued and positive. No solution exists, otherwise. 
The nonexistence of any solution for $b=c=0$ suggests that the function $K_{\alpha}\left(z\right)$, has no zero for $\left|\arg z\right|<\pi$, $z\neq 0$ and $1/2 >\beta>-1/2$.
As proof \cite{GHMC1967,HbkST}, relate the modified Bessel function to the Whittaker function, $W_{0, \beta}(z)= \left(z/\pi\right)^{1/2} K_{\beta}\left(z/2\right)$, and consider Eq. (\ref{EqWhittaker}).

\subsubsection{Lambert $W$ function}

The Lambert $W$ function is defined via the inverse of the map 
$W\to W\exp\left\{W\right\}$. The function is multivalued, has branches and several integral representations. We refer to \cite{W0,W1,W2,W3} for details. 
The transcendental equation 
\begin{eqnarray}
W\left(z\right)/z - b z-c=0,\hspace{1em}
\left|\arg z\right|<\pi, \label{W1}
\end{eqnarray}
has 
one solution 
under each of the following conditions,
\begin{eqnarray}
&& b>0,\hspace{1em} c<  \int_{1/e}^{\infty} \Im \left\{W(-t)\right\}/\left(\pi t^2\right) dt<e, \nonumber \\
&& b=0, \hspace{1em} 0<c<\int_{1/e}^{\infty} \Im \left\{W(-t)\right\}/\left(\pi t^2\right) dt <e. \nonumber 
\end{eqnarray}
 The solution is real valued and positive, $z=x_{16}$, and 
$$\left(e-c\right)/b>
\left(\int_{1/e}^{\infty} \Im \left\{W(-t)\right\}/\left(\pi t^2\right) dt-c\right)\Bigg/b
>x_{16}>0.$$
 No solution exists, otherwise. Since no solution exists for $b=c=0$, we recover \cite{W0} that the Lambert $W$ function has no zero for $\left|\arg z\right|<\pi$ and $z\neq 0$.
As proof \cite{W1}, consider the following relationship holding for $\left|\arg z\right|<\pi$, 
\begin{equation}\frac{W(z)}{z}= 
\int_{1/e}^{\infty}\frac{\Im \left\{W(-t)\right\}/\left(\pi t\right)}{t+z}\, dt, \label{intW}
\end{equation} 
the inequality $1>\Im \left\{W(-t)\right\}/\pi>0$, holding for $t>1/e$, and apply Theorem \ref{Th1}.

The transcendental equation 
\begin{eqnarray}
W^{\prime}\left(z\right) - b z-c=0,\hspace{1em}
\left|\arg z\right|<\pi, \label{W2}
\end{eqnarray}
has 
one solution 
under each of the following conditions,
\begin{eqnarray}
&& b>0,\hspace{1em} c<  \int_{1/e}^{\infty} \frac{1}{\pi t}\frac{d}{dt}\Im \left\{W(-t)\right\} dt, \nonumber \\
&& b=0, \hspace{1em} 0<c<\int_{1/e}^{\infty} \frac{1}{\pi t}\frac{d}{dt}\Im \left\{W(-t)\right\} dt. \nonumber 
\end{eqnarray}
 The solution is real valued and positive.
 No solution exists, otherwise.
The nonexistence of any solution for $b=c=0$ suggests that the derivative of the Lambert $W$ function, $W^{\prime}\left(z\right)$, has no zero for $\left|\arg z\right|<\pi$ and $z\neq 0$.
This property is expected since the following relationship \cite{W0} holds, $W^{\prime}\left(z\right)=W\left(z\right)/\left(z \left(W \left(z\right)+1\right)\right)$.
As proof \cite{W2,W3}, consider the following relationship holding for $\left|\arg z\right|<\pi$, \begin{equation}W^{\prime}(z)= 
\int_{1/e}^{\infty} \frac{1}{\pi(t+z)}\frac{d}{dt}\Im \left\{W(-t)\right\} dt, \label{intW'}
\end{equation} 
and apply Theorem \ref{Th1}.

\subsubsection{Binet function }

The Binet function $J(z)$ is defined as follows,
\begin{eqnarray}
J(z)=\log \left(\Gamma(z)\right)-\left(z-1/2\right)
\log \left(z\right)+z-\log \left(\left(2 \pi\right)^{1/2}\right).
\label{Jbinet}
\end{eqnarray}
The function is multivalued, has branches and is bounded in certain sectors of the complex plane. We refer to \cite{Hen2} for details.
The transcendental equation \begin{eqnarray}
J\left(z^{1/2}\right)\Big/ z^{1/2} -bz-c=0,
 \hspace{1em} \left|\arg z\right|<\pi,\label{EqJbinet}
\end{eqnarray}
in terms of the Binet function, has 
one solution if $b>0$, or if $b=0$ and $c>0$.
 The solution is real valued and positive. No solution exists, otherwise. 
The nonexistence of any solution for $b=c=0$ suggests that the Binet function $J\left(z\right)$ has no zero for $\left|\arg z\right|<\pi/2$ and $z\neq 0$.
  As proof \cite{Hen2,HbkST}, consider the following relationship,
	$$\frac{J\left(z^{1/2}\right)}{ z^{1/2}}=
	\int_0^{\infty} \frac{1}{2 \pi y^{1/2}\left(y+z\right)}\, \log\left(\frac{1}{1-e^{-2\pi y^{1/2}}}\right)\, dy,$$
	and apply Theorem \ref{Th1}. 
Notice that the critical value is infinite.

Starting from the above examples, the adoption of the operator $\mathcal{T}$, given by Eq. (\ref{T}), obviously provides infinitely many transcendental equations. These equations can be ascribed to the form (\ref{Eq}). Consequently, the conditions for existence, uniqueness, reality, positivity and bounds for the solution can be analyzed by evaluating the corresponding critical value. As the above examples show, for special functions that are the Stieltjes transforms of nonnegative functions, the present method provides a simple way to investigate the absence of zeros in certain sectors of the complex plane. 

\section*{Acknowledgements}
The author thankfully acknowledges Prof. F. Mainardi and Prof. S. Rogosin for the careful and critical reading of the manuscript and for the suggestion of additional references.

\end{document}